\documentclass[10pt]{article}     % -*- latex -*-
\usepackage{amsthm}
\usepackage{amsfonts, amssymb, amsmath, amscd, latexsym}
\usepackage{mathrsfs}
\usepackage{epsfig}
\usepackage[latin1]{inputenc}

\swapnumbers

\newtheorem{thm}{Theorem}[section]
\newtheorem{lemma}[thm]{Lemma}
\newtheorem{cor}[thm]{Corollary}

\newtheorem{remark}[thm]{Remark}

\newtheorem{definition}[thm]{Definition}

\numberwithin{equation}{section}

\usepackage{color}

\newcommand\bs[1]{{\boldsymbol{#1}}}

\def\RR{\mathbb R}
\def\ds{\displaystyle}
\def\R{\boldsymbol{R}}
\def\Q{\boldsymbol{Q}}
\def\la{\lambda}

\def\ep{\varepsilon}
\newcommand{\eu}{\textrm{e}}
\newcommand{\lir}{{\lim_{r \to \infty}}}
\newcommand{\liro}{{\lim_{r \to 0}}}
\newcommand{\lit}{{\lim_{t \to +\infty}}}
\newcommand{\lito}{{\lim_{t \to -\infty}}}
\newcommand\sol[3]{${\cal #1}\!\stackrel{{#2}}{\_}\!{\sf #3d}$}

\newcommand{\Rsol}{$\mathcal {R}$}
\newcommand{\Ssol}{$\mathcal {S}$}
\newcommand{\fdsol}{${\sf fd}$}
\newcommand{\sdsol}{${\sf sd}$}

\begin{document}
\title{On a diffusion model with absorption and production.}
\author {Matteo Franca\thanks{Dipartimento di Scienze Matematiche,
Universit\`a Politecnica delle Marche, Via Brecce Bianche 1, 60131 Ancona -
Italy. Partially supported by G.N.A.M.P.A. - INdAM (Italy) and
MURST (Italy)} \hspace{1mm} and
Andrea Sfecci
\thanks{Dipartimento di Scienze Matematiche,
Universit\`a Politecnica delle Marche, Via Brecce Bianche 1, 60131 Ancona -
Italy. Partially supported by G.N.A.M.P.A.}
}
  \maketitle
\pagestyle{myheadings}

\begin{abstract}
We discuss the structure of radial solutions of some superlinear elliptic equations which model
diffusion phenomena when both absorption and production are present.
 We focus our attention on solutions defined in $\RR$ (regular) or in $\RR\setminus\{0\}$
  (singular) which are infinitesimal at infinity, discussing also their asymptotic behavior.
  The phenomena we find are present only if absorption and production coexist, i.e., if
  the   reaction term changes sign.  Our results are then generalized to include  the case where Hardy potentials are considered.
\end{abstract}
\vspace{5mm} \textbf{Key Words:} $\,$ supercritical equations, Hardy potentials, radial
solution, regular/singular ground
states, Fowler transformation, invariant manifold.\\
\textbf{MR (2000) Subject Classification}: $\,$ 35j70, 35j10, 37d10.
   \vspace{5mm}
% \markboth{\textsc{author}}{\textsc{title}}
\section{Introduction}
In this paper we are interested in structure results for  radial solutions for a family of equations whose prototype has the following form
\begin{equation} \label{laplace} \tag{L}
\Delta u({\bf x})+ k(|\textbf{x}|) u|u|^{q-2}=0 \,, \qquad k(|\textbf{x}|)=
 \begin{cases}
K_1 & |\textbf{x}| \le 1 \\
 K_2   & |\textbf{x}| > 1
 \end{cases}
\end{equation}
where ${\bf x} \in
\mathbb{R}^n$, with  $n>2$, $q>2$, $K_1 K_2<0$;  we assume either  $K_1<0<K_2$ and $q>2^*$, or
$K_2<0<K_1$ and $2_*<q<2^*$, where  $2_*:=2\frac{n-1}{n-2}$
and $2^*:=\frac{2n}{n-2}$ are respectively the Serrin and the Sobolev critical exponent.

Since we just deal with radial solutions we will indeed consider the following singular ordinary differential equation
\begin{equation} \label{eqrad} \tag{Lr}
u'' +\frac{n-1}{r} u'  + k(r) u|u|^{q-2} = 0\,,
 \qquad k(r)=
 \begin{cases}
K_1 & r \le 1\,, \\
 K_2   & r > 1\,,
 \end{cases}
\end{equation}
where, abusing the notation, we have set $u(r)=u({\bf x})$ for $|{\bf x}|=r$, and
$'$ denotes differentiation with respect to $r$.

\medbreak

The interest in equations of the family~(\ref{laplace}) started long ago from
nonlinearities   where
 $k$ is a constant, either negative or positive, and then it was generalized to include the case where
 $k$ varies with $r$, thus finding several different possible structures for the solutions,
 see e.g. \cite{AQ,BJ2,B1,DN,F1,F6,Fprep,Fdie,JPY1,JK,Y1996} for a far from being exhaustive bibliography.
 Nowadays it has become a broadly studied   topic  and the discussion now includes a wide family of non-linearities see e.g.
 ~\cite{DF,GNN,JPY2,JPY3}  and references therein.
 Radial solutions play a key role for~\eqref{laplace}, since in many cases, e.g. $k(r)\equiv K>0$,
 positive solutions have to be radial (but also in many situations in which $k$ is allowed to vary, see e.g.~\cite{B1,GNN,LN}).
 They are also crucial to determine the threshold between fading and blowing initial data in the associated parabolic problem,
 see e.g.~\cite{DLL,W93}.

  It can be shown that,
 when $q>2$, positive solutions exhibit two behaviors as $r \to 0$
 and as $r \to \infty$.
 In particular, $u(r)$ may be
 a \emph{regular solution}, i.e. $u(0)=d>0$ and $u'(0)=0$,
 or a \emph{singular solution}, i.e. $\liro u(r)=+\infty$,
 a \emph{fast decay (f.d.) solution}, i.e. $\lir u(r) r^{n-2}=L$,
 or a \emph{slow decay (s.d.) solution}, i.e. $\lir u(r)r^{n-2}=+\infty$.

Moreover, a regular, respectively singular, positive solution $u(r)$ defined for any $r > 0$ such that $\lir u(r)=0$ is usually called \emph{ground states} (G.S.), resp. \emph{singular ground states} (S.G.S.).
 In the whole paper we use the following notation: we denote by $u(r,d)$ the regular solution of~(\ref{eqrad})
such that $u(0,d)=d$, and by $v(r,L)$ the fast decay solution such that
$\lir r^{n-2}v(r,L)=L$.

Equation~\eqref{laplace} is a widely studied topic and find many applications in different contexts. E.g., it can model
the equilibria for a nonlinear heat equation. In this case $u$ is the temperature and $k u|u|^{q-2}$ represents
 a termo-regulated reaction which produces heat  when $k>0$, or absorbs heat when $k<0$.

 It can also model the equilibria reached by a series of chemical reactions, see e.g.
 \cite{Mu2} for a derivation of the model (in particular Chapter~7, and especially~7.4), see also \cite[§1]{Mu1}. In this case $u$ represents the density of a substance~$A$
 reacting with   substrates $B$ and $D$ according to the following scheme:
 \begin{equation}
 \label{reaction}
 (q-1)A +B \underset{c_1}\to q A+C  \, , \qquad   (q-1)A +D \underset{c_2}\to E\,.
 \end{equation}
 In the first reaction $c_1$ we have $(q-1)$ particles $A$ which react with some substrate $B$ to produce $C$ and a larger number of particles of $A$ (say $q$ in this case). In the second, $c_2$,
 we have $(q-1)$ particles $A$ which react with some substrate $D$ to produce $E$
(in fact we obtain an equation of the same type also when  the substrates $B,D$ and
 the substance $C$ are not present).
 The two reactions can be modeled respectively
 by
\begin{equation}\label{parabolico}
\begin{array}{l}
u_t=\Delta u+ \nu_1 \mu_B \, u^{q-1} \,,\\
u_t=\Delta u- \nu_2(q-1)\mu_D \, u^{q-1} \,,
\end{array}
\end{equation}
where $\nu_1$ and $\nu_2$ are the velocities of the reactions,  and $\mu_B$,  $\mu_D$ are the density of the substrates which are assumed to be constant
(and they can be chosen to be $1$).
Here we are interested in the equilibria reached by $u$, assuming that we have diffusion ($\Delta u$), production inside a ball
(e.g., when $K_1=\nu_1 \mu_B>0$) and absorption outside (when $K_2=-\nu_2(q-1) \mu_D<0$), or the symmetric situation.
Equation \ref{parabolico} can model also  a series of unknown reactions starting
from the substances on the left of the arrow in \eqref{reaction} and ending with the ones on the right of the arrow:
Usually in chemistry   and especially in biochemistry
  we do not know all the intermediate steps which are actually taking place, so
  the models are constructed using just  the starting reagents and the final products we find.

A further simple  case (from the modeling viewpoint)
 we are able to deal with is the following type of reaction together
with its inverse.
$$ (q-1)A +B \underset{c_1}\to q A \, , \qquad   q A \underset{c_2}\to (q-1)A +B$$
In this case the corresponding model will be given by the following equations:
\begin{equation}\label{parabolico2}
\begin{array}{l}
u_t=\Delta u+ \nu_1 \mu_B \, u^{q-1} \,, \\
u_t=\Delta u- \nu_2 \, u^{q} \,,
\end{array}
\end{equation}
where again $\nu_1$ and $\nu_2$ are the velocities of the reactions,   $\mu_B$, is the density of the substrate which is assumed to be constant.
Again we are interested in the equilibria reached by $u$, assuming that we have, production inside a ball
(e.g., when $K_1=\nu_1 \mu_B>0$) and absorption outside (when $K_2=-\nu_2<0$), or the symmetric situation.

 As it is to be expected, solutions of~\eqref{eqrad} undergo to several bifurcations as $k$ changes sign and
 as $q$
  passes through
 some critical values, such as $2_*<2^*$. In particular all positive regular solutions are increasing and exist just in a ball of variable size if $k\equiv K<0$,
 while they are decreasing if $k\equiv K>0$; further G.S. with f.d. exist just for $q=2^*$.

As a first consequence of our main results we get the following.
\begin{cor}\label{cor1}
Assume   $q>2^*$, $K_1<0$ and $K_2>0$, in equation~\eqref{eqrad}; then there is a sequence  $D_k \nearrow D_\infty<\infty$ such that $u(r,D_0)$
is a G.S. with f.d., and for any $k \in \mathbb{N}$, $u(r,D_{k})$ is a regular-f.d. solution with
 exactly $k$ non degenerate zeroes. Moreover
 $u(r,d)$ is a G.S. with s.d. for any $0<d<D_0$, and
 $u(r,d)$ is a regular-s.d. solution with exactly $k$ non-degenerate zeroes
    whenever $D_{k-1}<d<D_{k}$, for any $k \ge 1$.
\end{cor}

\begin{cor}\label{cor2}
Assume   $2_*<q<2^*$, $K_1>0$ and $K_2<0$, in equation~\eqref{eqrad}; then there is a sequence $L_k \nearrow L_{\infty}< \infty$ such that $v(r,L_0)$ is a G.S. with f.d. and $v(r,L_{k})$ is a regular-f.d. solution with exactly $k$ non degenerate zeroes.
 Moreover, the fast decay solutions $v(r,L)$ are such that $v(r,L)$
is a  S.G.S. with f.d. for any $0<L<L_0$, and
 $v(r,L)$ is a singular-f.d. solution with exactly $k$ non-degenerate zeroes
    whenever $L_{k-1}<L<L_{k}$, for any $k \ge 1$.
Consequently there is a sequence $D_k \nearrow \infty$ such that $u(r,D_k)$ is a regular-f.d. solution with
exactly $k$ non-degenerate zeroes for any $k \ge 0$.
\end{cor}
\begin{remark}
Besides the fact that the nonlinearity in~\eqref{laplace} has a very special form, we believe that
it can be regarded as the prototype for a much wider class of equations, including smooth nonlinearities:
this will be the object of a forecoming paper. However the presence of negative $k(r)$ causes severe
technical difficulties, due to the lack of continuability of solutions $u(r)$ (in general they might be defined just in an annulus).
\end{remark}

  In this simple model, i.e. for~\eqref{laplace} and for~\eqref{hardy} below, we are able to solve completely all the
main questions:
 Corollaries \ref{cor1} and \ref{cor2}, and the more general results Theorems \ref{main1bis} and \ref{main2bis}, give  the exact structure for all the radial solutions of~\eqref{eqrad},
 classifying them according to their sign properties, and also precise asymptotic estimates of their asymptotic behavior, see Definition \ref{defRFS}, Remarks \ref{singandslow}, \ref{asymp.criticalbis}, \ref{singandslowbis}.
  Further we easily obtain precise results concerning  the relation between the values $K_i$
  (which represents the ratio between
  the velocity of the diffusion and the strength of the reaction), and the values and the positions  of the maxima
  of positive solutions, see Section~\ref{furtherrem}: this information might be of use in applications.

It is worthwhile to quote that in literature there are many results on the structure of radial solutions for Laplace equations
with indefinite weights $k$, even for more general potential, see e.g.~\cite{Bae2,BJ2,CFG}. However, these papers
are concerned with phenomena which are found when $k$ is a positive function,
and which persist even if
$k$ becomes      negative in some regions.
The structure results we find can just take place if we have a change in the
sign of $k$: if $q$ is either smaller or larger than $2^*$ there are no G.S. with fast decay, neither if $k(r) \equiv K>0$, nor if
  $k(r) \equiv K<0$. In fact, the structure of the solutions of~\eqref{eqrad} described in Corollaries~\ref{cor1} and~\ref{cor2}, reminds
  of the situation in which $q=2^*$ and we have a positive $k$ which behaves like a positive power for $r$ small and
  a negative power for $r$ large see e.g.~\cite{DF,Y1996}.
  In fact structure results which are typical of non-linearities with  sign-changing weights have been found in \cite{BGH}, but for
  bounded domains and just in the subcritical case (using a variational approach).  Further  in \cite{DMM,MaMa,Serena1} and in references therein
   the reader can find several nice and sharp structure results for sign-changing non-linearities, even for more general operators   ($p$-Laplace, relativistic and mean curvature),
   in the framework of oscillation (and non-oscillation) theory,
  but for exterior domains, i.e. for solutions defined, say for $r>1$.

\medbreak

In fact, our analysis is directly developed for a more general class of equations, including the singular term
$\eta u/r^2$, which usually takes the name of Hardy potential, and  a slightly larger class of nonlinearities. More precisely we consider
the following problem
\begin{equation}\label{hardy}\tag{H}
\Delta u+ \frac{\eta}{r^2}u + f(u,r)=0
\end{equation}
 or more precisely its radial counterpart
\begin{equation}\label{rad.hardy}\tag{Hr}
u''+ \frac{n-1}{r} u'+\frac{\eta}{r^2}u + f(u,r)=0\,,
\end{equation}
where, in the whole paper  we assume $\eta<\frac{(n-2)^2}{4}$, and we set
 \begin{equation}\label{types}
f(u,r)=
 \begin{cases}
K_1 r^{\delta_1} u |u|^{q_1-2} & r \le 1 \\
K_2 r^{\delta_2} u |u|^{q_2-2} & r>1 \,.
 \end{cases}
 \end{equation}
 We recall that $\frac{(n-2)^2}{4}$ is the best constant for Hardy inequality, and no positive solutions
 for~\eqref{rad.hardy} may exist close to $r=0$ for $\eta>\frac{(n-2)^2}{4}$, see e.g.~\cite{Cirstea}:
  this fact is reflected in a change of the stability properties
 of the origin of the dynamical system of Fowler type we are going to introduce in Section~\ref{FowlerT}.
 It is well known that the changes in the structure of positive solutions depend on the interaction
 between the exponents $q_i$, and
 the spatial inhomogeneities $r^{\delta_i}$ which determine a shift in the critical values for the exponents.
 For this purpose, following e.g.~\cite{F1}, we introduce the parameters
  \begin{equation}\label{elle}
 l_i= 2 \frac{q_i+\delta_i}{2+\delta_i} \, \qquad \textrm{for $i=1,2$} \,.
 \end{equation}
 Note that $l_i$ gives back $q_i$ if $\delta_i=0$, and that
 \eqref{hardy} with
 $f(u,r)= K r^{\delta_1} u|u|^{q_1-2}$ is subcritical, critical and supercritical iff
 $l_1$ is smaller, equal or larger than $2^*$, cf.~\cite{F1}.

Equation~\eqref{hardy} has been subject to deep investigation for different type of~$f$, see e.g. \cite{Bae1,FMT1,FMT,FGazz,Terr96}.
  The introduction of the
 singular Hardy terms affects deeply the asymptotic behaviour of the solutions, and the values of some critical exponents.
Therefore we need to relax the notion   of \emph{regular}   and \emph{fast decay} solutions: see Remark~\ref{differenze1} below. Notice in particular that no solutions bounded for $r>0$ can exist when $\eta>0$. However, using  generalized notion of regular and fast decay solutions (see Definition~\ref{defRFS})
we are able to prove Theorems~\ref{main1bis} and~\ref{main2bis} which extends Corollaries~\ref{cor1} and~\ref{cor2},
and the more general, but less precise Theorems~\ref{main1} and~\ref{main2}.

 We emphasize that we have assumed that $k(r)$ in~\eqref{eqrad} and $f(u,r)$ in~\eqref{types}
 changes discontinuously sign at $r=1$, just for simplicity: all the discussion
can be repeated in the case where $k(r)$ changes sign at $r=r_0>0$. In fact,  changing the spatial coordinate by setting $r=r_0 \, \tilde r$, we can pass  to an equation like~\eqref{laplace}
with $\tilde K_i= r_0^2 K_i$ and $\tilde r$ as independent variable.
A similar reasoning holds also for \eqref{rad.hardy}.

The proofs are based on the introduction of Fowler transformation~\cite{Fow}. This way the radial problems are reduced to
two different $2$-dimensional autonomous systems: one corresponding to $r \le 1$, and the other
to $r >1$. These problems are studied via invariant manifold theory, following
the way paved by \cite{JPY1,JPY2,JPY3,JK},
 and the structure results
for the original equations are found by a simple superposition of the two phase portraits.

We expect to find results analogous to Corollaries~\ref{cor1} and~\ref{cor2} also in the $p$-Laplace context, i.e. for radial solutions of
\begin{equation}\label{pLaplace}
\Delta_p u+   f(u,r)=0 \,,
\end{equation}
where $f$ is as in~\eqref{types},  $\Delta_p u=\textrm{div} (\nabla u |\nabla u|^{p-2})$ and $n>p$,
making use of the generalized Fowler transformation found in
\cite{Bv}. In order to avoid cumbersome technicalities we leave open this part suggesting an approach similar to the one adopted by the first author in~\cite{F6} (see also~\cite{Bv,F1,Fduegen}).

The paper is divided as follows: in Section~\ref{FowlerT} we introduce Fowler transformation, underlining the property of the corresponding dynamical system in presence of Hardy potentials. Then, in Sections~\ref{mag0} and~\ref{min0} we present the phase portrait analysis in the case of equation~\eqref{laplace} with $k(|x|)=K>0$ and $k(|x|)=K<0$. Section~\ref{mainthm} is devoted to the statement of the main results and their proofs. Finally, a slight generalization to other kind of nonlinearities will be presented in Section~\ref{slight}.
In Section~\ref{furtherrem}, the dependence of the maxima of the solutions $u(r,D_0)$ described in Theorem \ref{main1},
on the paprameters $K_i$ and $r_0$ of the problem is explored.

\section{Preliminaries: the autonomous case} \label{due}

\subsection{Fowler transformation and invariant manifolds}\label{FowlerT}
In this section we introduce a Fowler-type transformation (cf.~\cite{Fow}), which permits us to pass from equation~\eqref{rad.hardy} to a planar system.
In the whole section we assume
\begin{equation}\label{auto}
f(u,r)=K r^{\delta} u|u|^{q-2} \, , \qquad \delta>-2, \; q>2 , \quad l=  2 \frac{q+\delta}{2+\delta}>2 \,.
\end{equation}
We set $\alpha_l =\frac{2}{l-2}$ and $\gamma_l = \alpha_l+2-n$, and
\begin{equation}\label{transf}
\begin{cases}
x_l(t)=u(r)r^{\alpha_l} \\  y_l(t) =u'(r)r^{\alpha_l+1}
\end{cases}
\qquad \text{where } r=\eu^t \,,
\end{equation}
so that we pass from~(\ref{rad.hardy}) to the following autonomous system
\begin{equation}\label{sist}\tag{${\rm S}_l$} %sometimes it is quoted as (S_{2^*})
\left( \begin{array}{c}
\dot{x}_l \\
\dot{y}_l  \end{array}\right) = \left( \begin{array}{cc} \alpha_l &
1
\\ -\eta & \gamma_l
\end{array} \right)
\left( \begin{array}{c} x_l \\ y_l  \end{array}\right) +\left(
\begin{array}{c} 0 \\-
K x_l|x_l|^{q-2}\end{array}\right).
\end{equation}
We will draw the phase portraits of~\eqref{sist}, as $l$ varies and $K$ changes sign, see Figures~\ref{posportrait} and~\ref{negportrait}.
Notice that many of the results contained in these sections are well known in literature, but we collect them here for completeness.
Let us set
$$\kappa(\eta):= \frac{(n-2)- \sqrt{(n-2)^2-4\eta}}{2}\,. $$

If we linearize system~\eqref{sist} at the origin, we find two real distinct eigenvalues thanks to the assumption $\eta<\frac{(n-2)^2}{4}$,
i.e., $\lambda(l)=\gamma_l+\kappa(\eta)$ and
$\Lambda(l)=\alpha_l-\kappa(\eta)$, $\lambda(l)< \Lambda(l)$.
 The origin is a saddle if and only if $-\alpha_l \gamma_l  > \eta$ which corresponds to   $2_*(\eta)<l<{\rm I}(\eta)$,
%\ft{\begin{equation}\label{lh}
%2_*(\eta)<l<{\rm I}(\eta)\,,
%\end{equation}}
where
$$2_*(\eta):=2 \frac{n+\sqrt{(n-2)^2-4 \eta}}{n-2+\sqrt{(n-2)^2-4\eta}}  $$
(which gives back $2_*$ if $\eta=0$), and
$$
{\rm I}(\eta):= \begin{cases}
+\infty & \text{if } \eta\leq 0\\
2 \frac{n-\sqrt{(n-2)^2-4 \eta}}{n-2-\sqrt{(n-2)^2-4\eta}} & \text{if } 0<\eta<\frac{(n-2)^2}{4}\,.
\end{cases}
$$

\noindent
When $2_*(\eta)<l<{\rm I}(\eta)$, we find $\lambda(l)<0<\Lambda(l)$, where  $v_{\lambda(l)}=(1,-n+2+\kappa(\eta))$, $v_{\Lambda(l)}=(1,-\kappa(\eta))$
are the corresponding eigenvectors.
Hence  the origin is a saddle which admits a
$1$-dimensional unstable manifold $M^u$ which is tangent at the origin to the line $T^u$:
 $y=-\kappa(\eta) x$ and a $1$-dimensional stable manifold $M^s$
which is tangent at the origin to the line $T^s$: $y=-[n-2-\kappa(\eta)] x$.
Let $\phi(t,\Q)$ be the solution of~\eqref{sist} such that $\phi(0,\Q)=\Q$,  then we have the following basic correspondence
between~\eqref{sist} and~\eqref{rad.hardy}.
\begin{lemma}\label{manifold}
Assume $2_*(\eta)<l< I(\eta)$, and consider the trajectory $\phi(t,\Q)$ of~\eqref{sist},
and the corresponding solution $u(r)$ of~\eqref{rad.hardy}.

Then there are $D^{\infty} \in (0,\infty]$, $L^{\infty} \in (0,\infty]$
 such that
 \begin{eqnarray}
      \Q \in M^u & \Longleftrightarrow & \liro u(r)r^{\kappa(\eta)}=d \in (-D^{\infty},D^{\infty}) \label{corr1} \,,\\
       \Q \in M^s & \Longleftrightarrow & \lir u(r)r^{n-2-\kappa(\eta)}=L \in(-L^{\infty},L^{\infty})\label{corr2} \,.
 \end{eqnarray}
So we denote by $u(r,d)$ and $v(r,L)$ respectively the solutions in~\eqref{corr1}, and~\eqref{corr2}.
Then we have %\footnote{AS: qui diamo per scontato che $D^\infty >1$.}
\begin{equation}\label{scaling}
   u(r,d)=u(rd^{-1/(\alpha_l-\kappa(\eta))},1)d \, , \quad v(r,L)=v(rL^{1/[n-2-\kappa(\eta)-\alpha_l]},1)L\,.
\end{equation}
Further $D^{\infty}=L^{\infty}=+\infty$   if $K>0$ and they are both bounded if $K<0$.
\end{lemma}
\begin{proof}
The existence of the unstable manifold follows from standard facts in invariant manifold theory, cf. \cite[§ 13]{CodLev}; further if $\Q \in M^u$ then  $\phi(t,\Q) \eu^{-\Lambda(l) t} \to (d,-\kappa(\eta)d)$ as $t\to -\infty$, for a certain $d \in \RR$. Thus $u(r)r^{\kappa(\eta)} \to d$ as $r \to 0$.
Since~\eqref{sist} is autonomous, then $M^u$ is the graph of a trajectory.
So if $\phi(\tau,\Q)=\R$ and $v(r)$ is the solution corresponding to
$\phi(t,\R)$ then $\phi(t+\tau,\Q)=\phi(t,\R)$,
thus giving $v(\eu^{t+\tau},d)\eu^{\alpha_l \tau}=u(\eu^{t},1)$.
 Therefore
$$
\begin{aligned}
d &=\lim_{t \to -\infty} u(\eu^{t+\tau},d)\eu^{\kappa(\eta)(t+\tau)} \\ &=
\lim_{t \to -\infty}u(\eu^t,1)  \eu^{\kappa(\eta)t}\eu^{[\kappa(\eta)-\alpha_l]\tau} = \eu^{[\kappa(\eta)-\alpha_l]\tau} ,
\end{aligned}
$$
and the first equality in~\eqref{scaling} follows. Analogously we find
$$
\begin{aligned}
L=\lim_{t \to +\infty}
v(\eu^{t+\tau},L)\eu^{[n-2-\kappa(\eta)](t+\tau)} &=
\lim_{t \to +\infty} v(\eu^t,1)  \eu^{[n-2-\kappa(\eta)]t}\eu^{[n-2-\kappa(\eta)-\alpha_l]\tau} \\
&= \eu^{[n-2-\kappa(\eta)-\alpha_l]\tau} ,
\end{aligned}
$$
and the second equality  in~\eqref{scaling} follows, too.

It is well known that $D^{\infty}<\infty$ and $L^{\infty}<\infty$ if $K<0$ due to the lack of continuability of  solutions
of~\eqref{eqrad}, and~\eqref{rad.hardy}:
it can be shown by some Gronwall-type arguments. The fact that $D^{\infty}=+\infty=L^{\infty}$ if $K>0$
is again well established, see e.g.~\cite{CDZ}.
\end{proof}

Remark~\ref{manifold} provides a  smooth parametrization
 $\Psi^u:(-D^{\infty},D^{\infty}) \to M^u$ of $M^u$ and
 $\Psi^s:(-L^{\infty},L^{\infty}) \to M^s$ of $M^s$,
such that  $\Psi^u(0)=(0,0)=\Psi^s(0)$.

\begin{remark}\label{differenze1}
Notice that $u(r,d)$ is a regular solution and $v(r,L)$ is a fast decay solution
whenever $\eta=0$.
If  $0<\eta<\frac{(n-2)^2}{4}$ and $d>0$ then  $\kappa(\eta)>0$, thus $u(r,d)$
 is in fact singular, i.e. $\liro u(r)=+\infty$, and accordingly $u'(r)$ is negative and $\liro u'(r)=-\infty$
 as $r \to 0$. However if $\eta<0$ and $d>0$ then $\kappa(\eta)<0$, so that   $u(r,d) \to 0$ like a power
 as $r \to 0$, and it is monotone increasing for $r$ small.
\end{remark}

As a consequence we need to introduce the next terminology. We can recognize, if $\eta=0$, the usual notion of regular/singular and fast/slow decay solutions.

\begin{definition}\label{defRFS}
\begin{itemize}
\item A \Rsol-solution $u(r,d)$ satisfies $\liro u(r,d) r^{\kappa(\eta)} = d\in\RR$, while a \Ssol-solution $u$ satisfies $\liro u(r) r^{\kappa(\eta)} = \pm\infty$.
\item A \fdsol-solution $v(r,L)$ satisfies $\lir v(r,L)r^{n-2-\kappa(\eta)} = L\in\RR$, while a \sdsol-solution $u$ satisfies $\lir u(r)r^{n-2-\kappa(\eta)} = \pm \infty$.
\item a \sol Rkf $u(r,d)=v(r.L)$ is both a \Rsol-solution and a \fdsol-solution having $k$ nondegenerate zeros. Similarly we will treat \sol Rks $u(r,d)$, \sol Skf $v(r,L)$.
When we do not indicate the value $k$, e.g.  \sol S{}f,
we mean any solution with these asymptotic properties disregarding its number of zeroes.
\end{itemize}
\end{definition}

\medbreak

When $2<l < 2_*(\eta)$ the origin is an unstable node    for~\eqref{sist}, i.e., $\Lambda(l)>\la(l) >0$.
In this case we denote by $M^u$ the $1$-dimensional strongly unstable manifold, see \cite[§ 13]{CodLev},
which can be characterized as follows:

\begin{equation}\label{charac}
M^u:= \{\Q \mid \lito \|\phi(t,\Q)\| \eu^{- \Lambda(l) t} =c \in \RR \}\,.  %AS: cambiato in Lambda
\end{equation}

We emphasize that we have the same characterization for $M^u$ when $2_*(\eta)<l <I(\eta)$.
In fact the part of Remark~\ref{manifold} concerning $M^u$ continues to hold in this case too.
When $l=2_*(\eta)$ then  $\Lambda(l)>\lambda(l) =0$, so we have a central manifold: so $M^u$ is a classical unstable manifold and
satisfies Remark~\ref{manifold} and~\eqref{charac}.

Analogously, when $l > I(\eta)$ then $ \lambda(l)<\Lambda(l)< 0$, and when $l = I(\eta)$  $ \lambda(l)<\Lambda(l)=0$, so
the origin is respectively a stable node for~\eqref{sist} or it has a central and a stable direction.
So we denote by $M^s$ the $1$-dimensional stable ($l=I(\eta)$) or strongly  stable ($l>I(\eta)$) manifold, see \cite[§ 13]{CodLev},
\begin{equation}\label{characbis}
M^s:= \{\Q \mid \lit  \|\phi(t,\Q)\| \eu^{- \lambda(l) t} =c \in \RR\}\,. %AS: cambiato in lambda
\end{equation}
We emphasize that this last possibility does not take place when $\eta \le 0$, since
 ${\rm I}(\eta)=+\infty$ in this case. If $\eta>0$, the part concerning $M^s$ in Lemma~\ref{manifold} holds too.

Summing up we have the following.

\begin{lemma}\label{manifold2}
Assume that $2<l \le 2_*(\eta)$, and  $f$  is as in~\eqref{auto}.
Then there is a strongly unstable manifold $M^u$ (and no stable manifold), and Lemma~\ref{manifold} and Remark~\ref{differenze1}
continue to hold.

Analogously assume that $l \ge I(\eta)$, and  $f$ is as in~\eqref{auto}.
Then there is a strongly  stable manifold $M^s$ (and no unstable manifold), and Lemma~\ref{manifold} and Remark~\ref{differenze1}
continue to hold.
\end{lemma}

\begin{remark}\label{strong}
 We think it is worthwhile to notice that the behavior of \Rsol-so\-lu\-tions and \fdsol-solutions is the one ruled by the linear operator
 $\Delta u+ \frac{\eta}{r^2} u$, while the behavior of \Ssol-solutions and \sdsol-solutions depends mainly on the nonlinear term $u|u|^{q-2}$, at least
 if the nonlinearity is "strong enough". See Lemma~\ref{singandslow} below.
 \end{remark}

In the whole paper we denote by $M^{s,+}$, respectively by $M^{s,-}$, the branch of $M^s$ departing from the origin towards $x>0$, resp. $x<0$.
Similarly we denote by $M^{u,+}$, resp. $M^{u,-}$, the branch of $M^u$ departing from the origin towards $x>0$, resp. $x<0$.

Moreover, we denote by $T(m):= \{(x,y) \mid y=-mx \, , \; x >0 \}$ and by
 $T^{\pm}(m):= \{(x,y) \mid \pm(mx +y)>0 \, , \; x >0 \}$,
and recall that $M^u_+$ and $M^s_+$ are tangent respectively to $T^u=T(\kappa(\eta))$ and to $T^s=T(n-2-\kappa(\eta))$.
Observe further that the subset of the isocline $\dot{x}=0$ contained in $x>0$ lies
in $T^+(\kappa(\eta))$ iff $l>I(\eta)$, it is in the stripe between $T^u$ and $T^s$ iff
$l \in (2_*(\eta),I(\eta))$, it is in $T^-(n-2-\kappa(\eta))$ iff $2<l<2_*(\eta)$, it is tangent to $T^u$ and $T^s$ respectively iff
$l=I(\eta)$ and $l=2_*(\eta)$.

In the next subsections we turn to consider \Ssol-solutions and \sdsol-solutions: for this purpose we need to distinguish among the cases $K>0$ and $K<0$.

\subsection{Phase portraits of~\eqref{sist}  for $K>0$}\label{mag0}
When $K>0$ and $2_*(\eta)<l<I(\eta)$,~\eqref{sist} admits  two further nontrivial critical points $\boldsymbol{P^+} =(P_x,P_y)$ and $\boldsymbol{P^-} =(-P_x,-P_y)$ such that $P_x=[(-\alpha_l\gamma_l-\eta)/K]^{\frac{1}{q-2}}>0$ and $P_y=-\alpha_l P_x$. They are
 unstable for $2_*(\eta)<l<2^*$, centers if $l=2^*$, and stable for $2^*<l< {\rm I}(\eta)$.
 These critical points correspond  to \sol S0s $V(r)=P_x r^{-\alpha_l}$, and $-V(r)$.
  By symmetry, in what follows, we will focus our attention only on $\bs {P^+}$.

\medbreak

 From standard phase
plane arguments we get the following.

\begin{remark}\label{singandslow}
Assume $f$ is as in~\eqref{auto} and $K>0$.
\begin{enumerate}
  \item  If $2_*(\eta)<l<I(\eta)$,   there is at least a positive  to \sol S0s, $V(r)=P_x r^{-\alpha_l}$.

Further if $2_*(\eta)<l<2^*$ then the critical point $\bs {P^+}$ is unstable, so there is a two parameters  family
of trajectories $\phi(t,\Q)$ such that $\phi(t,\Q) \to \bs{P^+}$ as $t \to -\infty$. Therefore there is a two parameters  family
of  \Ssol-solutions of~\eqref{rad.hardy}, say $v(r)$, such that $v(r)r^{\alpha_l} \to P_x$ as $r \to 0$.

Dually if $2^*<l<I(\eta)$ then the critical point $\bs {P^+}$ is  stable, so there is a two parameters  family
of trajectories $\phi(t,\Q)$ such that $\phi(t,\Q) \to \bs {P^+}$ as $t \to +\infty$, and correspondingly a two parameters  family
of \sdsol-solutions of~\eqref{rad.hardy}, say $v(r)$, such that $v(r)r^{\alpha_l} \to  P_x$ as $r \to +\infty$.

  \item If $2<l<2_*(\eta)$, then the origin is the unique critical point and it is repulsive. However if $\phi(t,\Q) \to (0,0)$ as $t \to -\infty$
  but $\Q \not\in M^u$ then $\phi(t,\Q) \eu^{- \lambda(l) t} \to c(\Q) v_{\lambda(l)}$ as $t \to -\infty$, see~\cite{CodLev}. % cambiati due lambda
  So there is a two parameters family of \Ssol-solutions
  $v(r)$ of~\eqref{rad.hardy} which satisfies $v(r)r^{n-2-\kappa(\eta)} \to c(\Q) \in \RR \backslash \{0\}$ as $r \to 0$.

  \item If $l>I(\eta)$, then the origin is the unique critical point and it is attractive. If $\phi(t,\Q) \to (0,0)$ as $t \to +\infty$
  but $\Q \not\in M^s$, then $\phi(t,\Q) \eu^{-\Lambda(l)t} \to c(\Q) v_{\Lambda(l)}$ as $t \to +\infty$, see~\cite{CodLev}. % cambiati due Lambda
  So there is a two parameters family of \sdsol-solutions
  $v(r)$ of~\eqref{rad.hardy} which satisfies $v(r)r^{n-2-\kappa(\eta)} \to c(\Q) \in \RR \backslash \{0\}$ as $r \to +\infty$.
\end{enumerate}
\end{remark}
The critical cases $l=2_*(\eta),2^*,I(\eta)$ will be considered in Proposition~\ref{asymp.critical} and Remark~\ref{asymp.criticalbis} below.
%\ft{dopo nella prop si fara' anche una dim
%dei casi critici e si rilencheranno i comportamenti asintotici}

%\ft{le coordinate angolari le ho evitate. Torneranno insieme all'omotetia se non si riesce
%a riportare il caso $r=r_0$ ad $r=1$.\\
%   We introduce polar coordinates
%\begin{equation}\label{polcoord}
%\rho_{l}= \| \boldsymbol{x_l} \| \,, \qquad \phi_{l}= \arctan(y_{l}/ x_{l})\,.
%\end{equation}
%In fact most of the result of this subsection are well established facts but we collect them here and we give a sketch of the proof for completeness.}
The structure of radial solutions of~\eqref{laplace} is generally obtained using some Pohozaev type identity, see~\cite{Po}.
In this context we locate $M^u$ and $M^s$ through a Lijapunov function, which is the transposition in this context of
the Pohozaev function, see~\cite{MoSc}.
In fact, it is possible to obtain from~\eqref{sist} a second order differential equation
\begin{equation}\label{ODE}
\ddot x - (\alpha_l + \gamma_l) \dot x + (\alpha_l \gamma_l +\eta) x + K x |x|^{q-2} = 0\,,
\end{equation}
which suggests the introduction of the energy function
\begin{equation}\label{defE}
E(x,y)= \frac{(\alpha_l x +y)^2}{2} + (\alpha_l \gamma_l+\eta) \frac{x^2}{2}+K\frac{|x|^q}{q} \,.
\end{equation}
From a straightforward computation, we find
\begin{equation}\label{derivoE}
\frac{d}{dt}E(\dot{x}(t),x(t))= (\alpha_l+\gamma_l) \big(\alpha_l x(t) +y(t) \big)^2 \,.
\end{equation}
Notice that $\alpha_l+\gamma_l$ is  positive, null, or negative, respectively when $l<2^*$, $l=2^*$, $l>2^*$.
Further   the Poincar\'e-Bendixson criterion forbids the presence of periodic trajectories for $l \ne 2^*$.
%\footnote{la parte sopra e' comune: si potrebbe mettere all'inizio ma io la lascerei qui.}
Using this information it is possible to prove the following.
\begin{lemma}\label{picture}
 The phase portraits for~\eqref{sist} are as in Figure~\ref{posportrait} when $K>0$.
 The bifurcation diagram is sketched in Figure~\ref{biffig}.
 \end{lemma}

\begin{figure}[t!]
\centerline{\epsfig{file=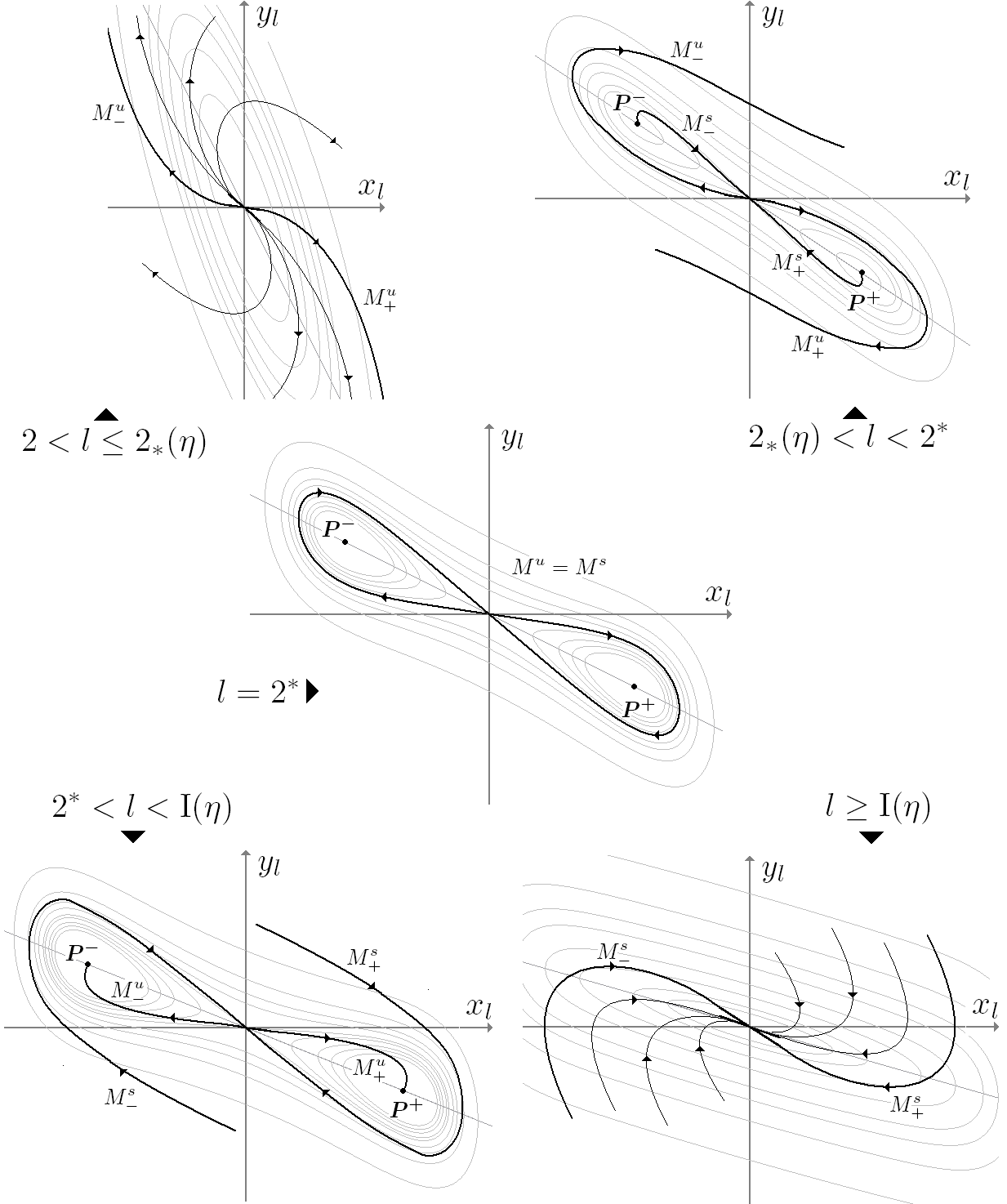, width = 12 cm}}
\caption{\small
The phase portrait of system~\eqref{sist}, for $K>0$, when $l>2$ varies. Some level curves of the energy $E$ defined in~\eqref{defE} are drawn. The energy $E$ is increasing for $l<2^*$ and decreasing for $l>2^*$. The origin is a global repeller for $2<l\leq 2_*(\eta)$ and we can identify the strongly unstable manifold $M^u$. If $2_*(\eta)<l<{\rm I}(\eta)$ the origin is a saddle and we can recognize an unstable manifold $M^u$ and a stable manifold $M^s$.
The nontrivial equilibria $\bs{P^\pm}$ have negative energy and are points of minimum. If $l=2^*$ the system is Hamiltonian and presents periodic orbits and two homoclinic trajectories, while if $l\neq 2^*$ there exists two heteroclinic trajectories. Finally, if $l\geq {\rm I}(\eta)$ the origin is a global attractor and we can identify the strongly stable manifold $M^s$.
}
\label{posportrait}
\end{figure}

 \begin{proof}
 We sketch the
argument for convenience of the reader.

$\bullet$ \textbf{Assume} $\bs{K>0}$ \textbf{and} $\bs{l \in (2_*(\eta), I(\eta))}$ (i.e. $\alpha_l \gamma_l +\eta<0$).\\
In this case the level sets of $E$ are bounded, and the $0$-set is a
$8$ shaped curve and the origin is the ``junction of the $8$'': $E$ is negative inside and positive outside.
There are two critical points $\bs{P^+}$ and $\bs{P^-}$ with $E(\bs{P^+})=E(\bs{P^-})<0$.
Observe further that if a trajectory is unbounded, either in the past or in the future, then it has to cross the coordinate axis indefinitely:
 it simply depends on the fact that the nonlinear term $K x|x|^{q-2}$
is ruling for $x$ large enough. By the way
this can be proved adapting the argument of
\cite{Fdie}, i.e., introducing polar coordinates and studying the angular velocity of ``large'' solutions.

Using this information we see that
\begin{enumerate}
    \item If $2_*(\eta)<l<2^*$ then $M^s$ is made by  two paths $M^s_+$ and $M^s_-$, from  the critical points $\bs{P^+}$ and $\bs{P^-}$ respectively
   to the origin; $M^s_+$ is in $\{E<0\}\cap\{x \ge 0\}$ and $M^s_-$ is
   obtained by symmetry. $M^u$ is an unbounded double spiral rotating from the origin clockwise.
     \item If $l=2^*$ then $E$ is a first integral so  $M^u=M^s$ and they are the union of two homoclinic, one contained in $x>0$ one in $x<0$.
Inside and outside there are periodic trajectories.
  \item If $2^*<l<I(\eta)$ then $M^u$ is made by  two paths $M^u_+$ and $M^u_-$, from the origin to the critical points $\bs{P^+}$ and $\bs{P^-}$ respectively; $M^u_+$ is in $\{E<0\}\cap\{x \ge 0\}$ and $M^s_-$ is symmetric. $M^s$ is an unbounded double spiral rotating from the origin counter-clockwise.
\end{enumerate}
 $\bullet$ \textbf{Assume} $\bs{K>0}$ \textbf{and} $\bs{l \in (2,2_*(\eta)] \cup [I(\eta),+\infty))}$ (i.e. $\alpha_l \gamma_l +\eta>0$).\\
Then  the level sets of $E$ are bounded concentric curves, centered in the origin. The origin is the unique critical point and it is
an unstable node if $2<l < 2_*(\eta)$ (and a stable node if $ l > I(\eta)$),   and it has a center manifold for $l=2_*(\eta),I(\eta)$.
Therefore,  if $2<l \le 2_*(\eta)$ (respectively if $ l \ge I(\eta)$), all the trajectories are unbounded spiral crossing indefinitely the
 coordinate axes, converging to the origin as
$t \to -\infty$ (resp. as $t \to +\infty$) and unbounded as $t \to +\infty$ (resp. as $t \to -\infty$). In particular this holds for trajectories
of the strongly unstable manifold $M^u$ (resp. the strongly  stable manifold $M^s$).
\end{proof}
Now we go back to consider the asymptotic behavior of \Ssol-solutions and \sdsol-solutions
in the critical cases $l=2_*(\eta),I(\eta)$, and then $l=2^*$. The proof of the following Lemma is an adaption of the proof of  Corollary 2.5, developed in \cite[Appendix]{Fduegen}, where it is worked out in the $p$-Laplace context.

\begin{lemma}\label{asymp.critical}
Assume $K>0$ and $l=2_*(\eta)$, let $\Q \not\in M^u$. Consider the trajectory $\phi_l(t,\Q)$ of~\eqref{sist} and the corresponding solution
$v(r)$ of~\eqref{rad.hardy}. Then $ v(r)r^{n-2-\kappa(\eta)} |\ln (r)|^{\frac{1}{(q-2)}}$ is bounded between two positive constants as $r \to 0$.

Assume $K>0$ and $l=I(\eta)$, let $\R \not\in M^s$. Consider the trajectory $\phi_l(t,\R)$ of~\eqref{sist} and the corresponding
solution $v(r)$ of~\eqref{rad.hardy}. Then $ v(r)r^{n-2-\kappa(\eta)} [\ln (r)]^{\frac{1}{(q-2)}}$ is bounded between two positive constants as $r \to \infty$.
\end{lemma}

\begin{proof}
Assume first $\eta=0$, $l=2_*(0)$: in this case,~\eqref{sist} is
\begin{equation}\label{ss}
\begin{cases}
\dot x = A x + y\\
\dot y = - K x |x|^{q-2}
\end{cases}
\end{equation}
with $A=n-2$. Let $\Q \not\in M^u$, then by Lemma~\ref{picture} we see that
$\lito \phi_l(t,\Q)=(0,0)$ (this depends on the fact that $E$ is increasing along the trajectories). Let $\phi_l(t,\Q)=(X(t),Y(t))$, and assume to fix the ideas that $X(t)>0$ for $t\ll0$. {From standard tools in invariant manifold theory, see e.g.~\cite{CodLev},
we see   that $\phi_l(t,\Q)$  approaches the line $T(A)=T(n-2)$ as $t \to -\infty$ (i.e. the central direction)} and that it converges to $(0,0)$ polinomially.

We claim  that \textbf{there is $T \in \RR$ such that   $\phi_l(t,\Q) \in T^+(A)$
for $t< T$.}

In fact the flow  of~\eqref{sist} on $T(A)$ points towards  $T^-(A)$, therefore we assume there is $\tau \in \RR$ such that  $\phi_l(t,\Q) \in T^-(A)$ for any $t< \tau$ and $X(\tau)\geq 0$ (otherwise
%\ft{ there is $T \in \RR$ such that  $\phi_l(t,\Q) \in T^+(\alpha_l)$ for any $t<T$ and  $\phi_l(T,\Q) \in T(\alpha_l)$ (and the claim is proved)}
the claim is proved).
Then $\dot{X}(t)<0<X(t)$ for $t< \tau$, but $X(t)=\int_{-\infty}^t\dot{X}(s)ds<0$ which gives a contradiction and proves the claim .

\medbreak

It follows that there is $t_0<T$ such that $\dot{X}(t)>0$ and $Y(t)=-A X(t)+ h(X(t))$, where $h(x)=o(x)$ as $x \to 0$ and it is positive. In particular for any $\ep>0$ we can choose
$t_0=t_0(\ep)$ such that
\begin{equation}\label{diseq}
(A-\ep) X(t)<-Y(t)<A X(t)
\end{equation}
for $t< t_0$. Further we see that $Y<0$, $\dot Y<0$ for $t<t_0$. So setting $Z(t)=-Y(t)$,
we get
 $\dot Z = K X^{q-1}$, and from ~\eqref{diseq} we find
$0<M_1  < Z^{1-q}(t) \dot{Z}(t)  < M_2$ for some suitable constants $M_i$.
So, integrating and using~\eqref{diseq},
we can find some positive constants $C_i$ such that
  \begin{equation}\label{stimaperX}
      C_1(1+M_1 |t|)^{-1/(q-2)} \le X(t) \le C_2(1+M_2 |t|)^{-1/(q-2)}
  \end{equation}
  for $t<t_0$, and   part of the lemma concerning the case $l=2_*(0)$ follows.
  Observe now that we have proved the result for an equation of the form
  \begin{equation}\label{ddotbis}
    \ddot{X}=A \dot X-K X|X|^{q-2}
  \end{equation}
  where, in this case, $K>0$ and $A=\alpha_l=n-2$. Notice that we can let $A$ be any positive constant
  and the proof still goes through.

  Now assume $\eta <\frac{(n-2)^2}{4}$ and $l=2_*(\eta)$; then $\alpha_l=n-2-\kappa(\eta)$. Consider the trajectory $\phi_l(t,\Q)=(X(t),Y(t))$ of~\eqref{sist}  with $\Q \not\in M^u$, and the corresponding solution $v(r)$ of~\eqref{rad.hardy}.
  Now, we have $A=\alpha_l+\gamma_l=n-2-2\kappa(\eta)=\sqrt{(n-2)^2-4\eta}$ and  $\alpha_l \gamma_l + \eta =0$  in~\eqref{ODE}, thus giving a differential equation as in~\eqref{ddotbis}. Introducing the variables $\mathcal X(t) = X(t)$ and $\mathcal Y(t) = Y(t) - \gamma_l X(t)$ we obtain a system as in~\eqref{ss}.
  Hence, arguing as above, $\phi_l(t,\Q)$ converges to the origin polynomially; repeating the previous argument we find
   again estimates as in~\eqref{stimaperX}, and consequently for $r< \eu^{t_0}$ we get
     \begin{equation}\label{stimaperX2}
      C_1(1+M_1 |\ln(r)|)^{-\frac{1}{q-2}} \le v(r)r^{n-2-\kappa(\eta)} \le C_2(1+M_2 |\ln(r)|)^{-\frac{1}{q-2}} \,.
  \end{equation}

  When $l=I(\eta)$,   we find that the origin is stable even in its central direction,
   and the Lemma can be obtained reasoning as above but reversing time.
\end{proof}
%Therefore for $l=2_*(\eta)$ and $l=I(\eta)$ we have a two parameters family  respectively of singular and slow decay solutions with the behavior described in Lemma~\ref{asymp.critical}.
Using the fact that for $l=2^*$ the bounded set enclosed by the homoclinic trajectory is filled by periodic solutions we get the following.
\begin{remark}\label{asymp.criticalbis}
Assume $l=2^*$, then there is a positive \sol S0s  $V(r)=P_x r^{-\alpha_l}$ and a two parameter family of \Ssol-solutions $W(r)$ such that $W(r)r^{\alpha_l}$ is uniformly positive and bounded for any $r>0$. All these solutions are in fact positive \sol S0s.
\end{remark}

\subsection{Phase portraits of~\eqref{sist}  for $K<0$}\label{min0}

\begin{figure}[t!]
\centerline{\epsfig{file=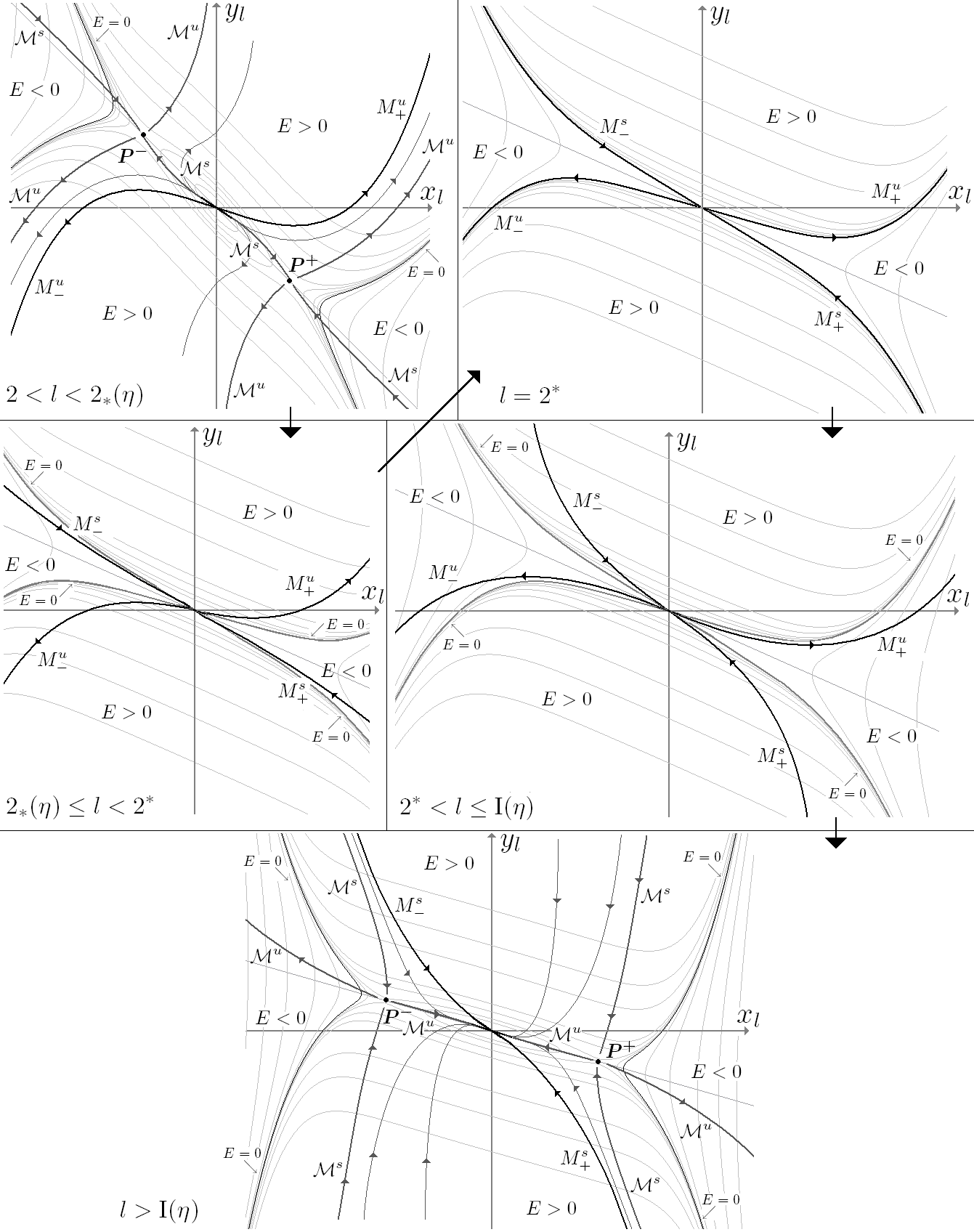, width = 12 cm}}
\caption{\small
The phase portrait of~\eqref{sist}, for $K<0$, when $l>2$ varies. Some level curves of the energy $E$ are drawn. The origin is the unique equlibrium if $l\in[2_*(\eta), {\rm I}(\eta)]$ (the system is Hamiltonian if $l=2^*$). If $l< 2_*(\eta)$ the origin is repulsive and if $l>{\rm I}(\eta)$ it is attractive. In these cases, there exist two non-trivial equilibria which are saddles. We identify their unstable and stable manifolds respectively with $\mathcal M^u$ and $\mathcal M^s$. The origin is a local minimum of the energy $E$, which is increasing for $l<2^*$ and decreasing for $l>2^*$. Notice the heteroclinic connections in these cases.
}
\label{negportrait}
\end{figure}

Also in this case, for some values of the parameters,~\eqref{sist} admits the critical points
$\bs{P^\pm}=(\pm P_x,\mp \alpha_l P_x)$ where $P_x= \left[(-\alpha_l \gamma_l - \eta)) /K\right]^{\frac{1}{q-2}}$.
\begin{remark}\label{singandslowbis}
Assume $f$ is as in~\eqref{auto} and $K<0$.
\begin{enumerate}
 \item If $2_*(\eta)\le l \le I(\eta)$, then system~\eqref{sist} admits no critical points.
 \item If either $2<l<2_*(\eta)$ or $l>I(\eta)$, then there is a critical point $\bs{P^\pm}$ and correspondingly  a positive \sol S0s $V(r)=P_x r^{-\alpha_l}$. Further   $\bs{P^\pm}$ is  a saddle, so there
  are two $1$-parameter families respectively of \Ssol-solutions $v(r)$, and of \sdsol-solutions $w(r)$ such that $v(r)r^{\alpha_l} \to P_x$
 as $r \to 0$ and
 $w(r)r^{\alpha_l} \to P_x$  as $r \to +\infty$.
\end{enumerate}
\end{remark}

A simple computation gives the following.

\begin{lemma}\label{kmeno}
For any $l>2$ the flow on~\eqref{sist} on $T(m)$ points towards $T^+(m)$ whenever $m \in [\kappa(\eta), n-2-\kappa(\eta)]$.
Further, if either  $m  <\kappa(\eta)$, or $m> n-2-\kappa(\eta)$, then there is $S(m)$ such that the flow on~\eqref{sist} on $(x,-mx) \in T(m)$ points towards $T^+(m)$ iff $x>S(m)$, it is tangent to $T(m)$ if $x=S(m)$ and points towards $T^-(m)$ iff $0<x<S(m)$.
\end{lemma}
\begin{proof}
From a straightforward computation we see that  if $\Q=(X,-m X ) \in T(m)$, and $\phi(t,\Q)=(x(t),y(t))$ is the corresponding trajectory of~\eqref{sist}, then
$$\dot{y}(0)+m \dot{x}(0)= -X [m^2- (n-2)m+ \eta] -KX^{q-1} \,.$$
Since $m^2- (n-2)m+ \eta \leq 0$  iff $m \in [\kappa(\eta), n-2-\kappa(\eta)]$, and $q>2$ the thesis follows.
\end{proof}

We can now draw the phase portrait.

\begin{lemma}\label{pictureneg}
 The phase portraits for~\eqref{sist} are as in Figure~\ref{negportrait} when $K<0$.
  The bifurcation diagram is sketched in Figure~\ref{biffig}.
 \end{lemma}
 \begin{proof}

$\bullet$ \textbf{Assume} $\bs{K<0}$ \textbf{and} $\bs{l \in [2_*(\eta), I(\eta)]}$ (i.e. $\alpha_l \gamma_l +\eta \le 0$).\\
The level sets of $E$ are unbounded   curves (hyperbola like), and the origin is the unique critical point. The origin is a saddle
if $2_*(\eta)<l< I(\eta)$. If $l=2_*(\eta)$ there is a $1$-dimensional unstable manifold $M^u$ and a $1$-dimensional center manifold,
say $M^s$, which is in fact asymptotically  stable: this fact can be easily obtained observing that $\dot y>0$ in $M^s \cap \{x>0 \}$.
Notice however that for $M^s$~\eqref{characbis} does not hold, and trajectories in $M^s$ behave polynomially.
Dually if $l=I(\eta)$ there is a $1$-dimensional stable manifold $M^s$ and a $1$-dimensional center manifold, say $M^u$,
 which is in fact asymptotically  unstable (for which however~\eqref{charac} does not hold, and we have a polynomial behavior).

In all the cases, by Lemma~\ref{kmeno}, $M^u$ is an unbounded curve,
 which crosses the $x$ axis at most once (if $\eta>0$, never if $\eta \le 0$), it is  in $y>0$ for $x$ large,
and   $M^u_+ \subset T^+(\kappa(\eta))$. It follows that $\phi_l(t,\Q)=(x_l(t,\Q),y_l(t,\Q))$ is such that
$\dot{x}_l(t,\Q)>0$ for any $t$, if $\Q\in M^u_+$: therefore $M^u_+$ is a graph on the $x>0=y$ semi-axis. $M^u_-$ is obtained by symmetry.
 Similarly
 $M^s$ is an unbounded curve, $M^s_+ \subset T^-(n-2-\kappa(\eta))$, and $M^s_-$ is obtained by symmetry.
Further $\phi(t,\Q)$ is defined just for $t \in (-\infty,T(\Q))$ and becomes unbounded as $t \to T(\Q)$  whenever $\Q \in M^u$;
it is defined just for $t \in (\tau(\Q),+\infty)$ and becomes unbounded as $t \to \tau(\Q)$  whenever $\Q \in M^s$, and
just for $t \in (\tau(\Q),T(\Q))$ and becomes unbounded as $t \to \tau(\Q)$ and as $t \to T(\Q)$ whenever $\Q \not\in (M^u\cup M^s)$,
where $\tau(\Q),T(\Q) \in \RR$.

$\bullet$ \textbf{Assume} $\bs{K<0}$ \textbf{and} $\bs{l \in (2,2_*(\eta)) \cup (I(\eta),+\infty))}$ (i.e. $\alpha_l \gamma_l +\eta>0$).\\
There are two critical points $\bs{P^+}$ and $\bs{P^-}$ with $E(\bs{P^+})=E(\bs{P^-})>0$. The origin is a node, unstable if $l \in (2,2_*(\eta))$ and stable if
$l>I(\eta)$, while $\bs{P^+}$ and $\bs{P^-}$ are saddle.

If $l \in (2,2_*(\eta))$ as in the previous case $M^u_+$ is an unbounded curve which crosses the $x$ axis at most once (if $\eta>0$, never if $\eta \le 0$), it is  in $y>0$ for $x$ large; further   $M^u_+ \subset T^+(\kappa(\eta))$ so $M^u_+$ is a graph on the $x>0=y$ semi-axis.
$M^u_-$ is obtained by symmetry. So, if $\Q\in M^u$ then $\phi(t,\Q)$ converges to the origin as $t \to-\infty$ and becomes
unbounded at some finite $t=T(\Q)$ (so it is defined just for $t<T(\Q)$). Further there is $\R$ such that
$\phi(t,\R) \to (0,0)$ as $t \to -\infty$, $\phi(t,\R) \to \bs{P^+}$ as $t \to +\infty$,
and  $\phi(t,\R) \in  T^-(n-2-\kappa(\eta))$  for any $t \in \RR$. In particular $\phi(t,\R)$ is a graph on $y=0$.
 By symmetry we also have  a heteroclinic connection between the origin and $\bs{P^-}$ which is again a graph  on $y=0$.

If $l>I(\eta)$ and $\Q\in M^s$ then $\phi(t,\Q)$ converges to the origin as $t \to +\infty$ and becomes
unbounded going backward in time at some finite $t=\tau(\Q)$ (so it is defined just for $t>\tau(\Q)$).  Further there is $\R$ such that
$\phi(t,\R) \to \bs{P^+}$ as $t \to -\infty$, $\phi(t,\R) \to (0,0)$ as $t \to +\infty$,
and  $\phi(t,\R) \in  T^+(\kappa(\eta))$  for any $t \in \RR$. In particular $\phi(t,\R)$ is a graph on $x>0=y$.
Again we also have  a heteroclinic connection between $\bs{P^-}$ and the origin which is again a graph  on $y=0$.
\end{proof}

 \begin{figure}[t]
\centerline{\epsfig{file=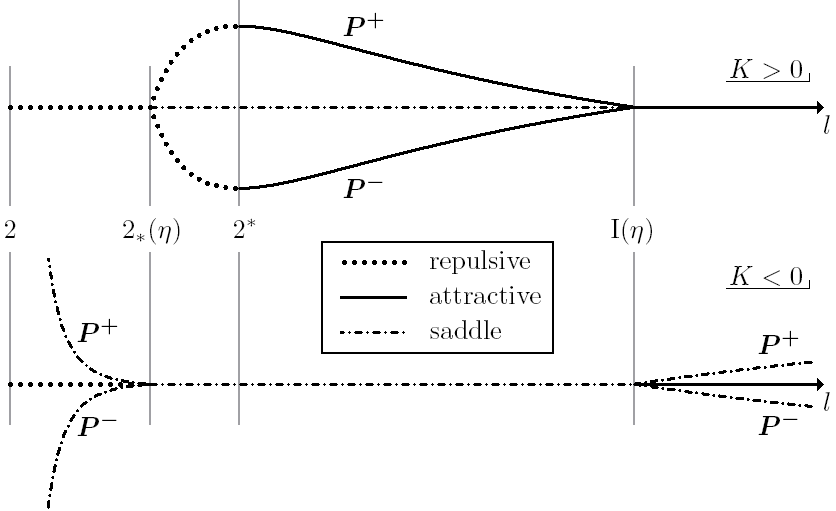, width = 8 cm}}
\caption{\small
The diagrams show how the origin bifurcates in two non-trivial equilibria at $l=2_*(\eta)$ and $l={\rm I}(\eta)$. In the case $K>0$ the non-trivial equilibria change their behaviour at $l=2^*$ (where they are centers).
% In the case $K<0$ the origin is a saddle for $2_*(\eta)<l<{\rm I}(\eta)$
% and bifurcates into two non-trivial saddles getting an attractive/repulsive equilibria
}
\label{biffig}
\end{figure}

We   conclude the subsection with the analysis of the asymptotic behavior of the trajectories in the centre manifolds found in the critical cases $l=2_*(\eta)$ and $l=I(\eta)$.

\begin{lemma}\label{asymp.criticalmeno}
Assume $K<0$ and $l=2_*(\eta)$, let $\Q$ belong to the centre manifold $M^s$. Consider the trajectory $\phi_l(t,\Q)$ of~\eqref{sist} and the corresponding solution
$v(r)$ of~\eqref{rad.hardy}. Then $ v(r)r^{n-2-\kappa(\eta)} [\ln (r)]^{\frac{1}{(q-2)}}$ is bounded between two positive constants as $r \to +\infty$.

Assume $K<0$ and $l=I(\eta)$, let  $\bs{R}$ belong to the centre manifold $M^u$. Consider the trajectory $\phi_l(t,\bs{R})$ of~\eqref{sist} and the corresponding
solution $v(r)$ of~\eqref{rad.hardy}. Then $ v(r)r^{n-2-\kappa(\eta)} |\ln (r)|^{\frac{1}{(q-2)}}$ is bounded between two positive constants as $r \to 0$.
\end{lemma}
\begin{proof}
Assume $l=2_*(\eta)$; since $\dot{y}>0$ along $M^{s}_+$,
we see that $\phi_l(t,\Q)=(X(t),Y(t))$ converges to the origin (polynomially fast) as $t \to +\infty$. Therefore, reasoning as in the proof of Lemma~\ref{asymp.critical},
we see that $\dot{X}(t)<0$ for $t$ large. Hence, repeating the argument in the proof of Lemma~\ref{asymp.critical},
we end up with the estimate~\eqref{stimaperX} for $t$ large. So the estimate for $v(r)$ easily follows.

When   $l=I(\eta)$ since $\dot{x}>0$ along $M^{u}_+$,
we see that $\phi_l(t,\Q)=(X(t),Y(t))$ converges to the origin (polynomially fast)  as $t \to -\infty$. Then we conclude reasoning as above but reversing $t$.
\end{proof}

\section{Main theorems: eq~\eqref{rad.hardy} with $f$ of type~\ref{types}}\label{mainthm}

In this section we consider equation~\eqref{rad.hardy} with $f$ of type~\eqref{types}, and $l_1,l_2>2$ are the values defined by
\eqref{elle}. In the following statements we   present the results for solutions which are {\em positive near zero}. The counterpart for {\em negative near zero} solutions follows by symmetry.
We adopt the terminology introduced in Definition~\ref{defRFS}.

\begin{thm}\label{main1}
Let $f$ be of type~\eqref{types} with $K_1<0<K_2$, $2<l_1<I(\eta)$, $l_2>2^*$;
% and $l_1 \le l_2$
then there is a sequence  $D_k \nearrow \infty$ such that for any $k \in \mathbb{N}$, $u(r,D_{k})$ is a \sol Rkf.
 %\RF-solution with exactly $k$ non degenerate zeroes
 Moreover,
  $u(r,d)$ is a positive \sol R0s for any $0<d<D_0$,
and for any $k \ge 1$  there exists $\tilde D_k \in [D_{k-1},D_k)$ such that
$u(r,d)$ is a \sol Rks %\RS-solution with exactly $k$ non-degenerate zeroes
    whenever $\tilde D_{k }<d<D_{k}$ and $u(r,\tilde D_k)$ is a \sol R{k-1}f.
\end{thm}

If we add the assumption $l_1\le l_2\leq {\rm I}(\eta)$ we find $\tilde D_k=D_{k-1}$ thus giving more structure.

\begin{thm}\label{main1bis}
Assume that we are in the hypothesis of Theorem~\ref{main1}; assume further $l_1\leq l_2\leq {\rm I}(\eta)$.
Then  $u(r,d)$ is   a \sol Rks for any $D_{k-1}<d<D_k$ for any $k\geq 1$.
%
% Then again
%there is a sequence  $A_k \nearrow \infty$ such that for any $k \in \mathbb{N}$, $u(r,A_{k})$ is a \RF-solution with
% exactly $k$ non degenerate zeroes.
% However $u(r,d)$ is   a positive \RS-solution    for any $0<d<A_0$, and
% a \RS-solution with exactly $k$ non-degenerate zeroes
%    whenever $D_{k-1}<d<A_{k}$, for any $k \ge 1$.
\end{thm}
We obtain the following {\em dual} result, too.

\begin{thm}\label{main2}
Let $f$ be of type~\eqref{types} with
 $K_2<0<K_1$, $2<l_1<2^*$, $l_2>2_*(\eta)$; then there is a sequence $L_{k} \nearrow +\infty$ such that for any $k\in\mathbb N$, $v(r,L_{k})$
  is a \sol Rkf.
 Moreover, $v(r,L)$ is   a positive \sol S0f for any $0<L<L_0$,
and for any $k \ge 1$  there exists $\tilde L_{k-1} \in [L_{k-1},L_{k})$ such that
$v(r,L)$ is a \sol Skf
    whenever $\tilde L_{k-1 }<L<L_{k}$ and $u(r,\tilde L_{k-1})$ is a \sol R{k-1}f.

Consequently there is a sequence $D_k\nearrow \infty$ such that $u(r,D_k)$ is a \sol Rkf for any $k\geq 0$.
\end{thm}
With an additional assumption on $l_i$, we obtain a better comprehension of the  structure as in Theorem~\ref{main1bis}.

\begin{thm}\label{main2bis}
Assume that we are in the hypothesis of Theorem~\ref{main2}; assume further $l_2 \ge l_1 \geq 2_*(\eta)$.
Then  $v(r,L)$ is   a \sol Skf for any $L_{k-1}<L<L_{k}$.
\end{thm}
\begin{remark}
Notice that the asymptotic behaviour of the \sdsol -solutions described in  Theorem~\ref{main1}
     changes when $l_2$ passes through the critical value~$I(\eta)$, cf. Remark~\ref{singandslow}.
     Similarly the asymptotic behaviour of the \Ssol -solutions described in  Theorem~\ref{main1}
     changes when $l_1$ passes through the critical value~$2_*(\eta)$, cf. again Remark~\ref{singandslow}.
\end{remark}

\subsection{Proof of the main results.}
To prove
Theorems~\ref{main1},~\ref{main1bis},~\ref{main2} and~\ref{main2bis} we need to
overlap the manifold $M^u$ obtained for $l=l_1$ with the manifold $M^s$ obtained for $l=l_2$.
Therefore, following~\cite{Fprep} we introduce the new variables
$$
x_{l_*}(t)=
\begin{cases}
u(\eu^t) \eu^{\alpha_{l_1}t} & t  \le 0 \\
u(\eu^t) \eu^{\alpha_{l_2}t}& t \geq 0 \,,
\end{cases} \quad  y_{l_*}(t)=
\begin{cases}
u'(\eu^t) \eu^{(\alpha_{l_1}+1)t} & t  \le 0 \\
u'(\eu^t) \eu^{(\alpha_{l_2}+1)t}& t \geq 0 \,,
\end{cases}
$$
so that~\eqref{rad.hardy} is changed into
\begin{equation}\label{sist*}
\left( \begin{array}{c}
\dot{x}_{l_*} \\
\dot{y}_{l_*}  \end{array}\right) = \left( \begin{array}{cc} \alpha_{l_*} &
1
\\ -\eta & \gamma_{l_*}
\end{array} \right)
\left( \begin{array}{c} x_{l_*} \\ y_{l_*}  \end{array}\right) +\left(
\begin{array}{c} 0 \\-
g_{l_*}(x,t)\end{array}\right)
\end{equation}
where
$$
g_{l_*}(x,t)=
\begin{cases}
K_1 x|x|^{q_1-2} & t \le 0\\
K_2 x|x|^{q_2-2} & t \ge 0
\end{cases}\,, \qquad
\alpha_{l_*}(t)=
\begin{cases}
\alpha_{l_1} & t \le 0\\
\alpha_{l_2} & t \ge 0 \,, \qquad
\end{cases}
$$
and $\gamma_{l_*}(t)=\alpha_{l_*}(t)-n+2$.
Hence, at the time $\tau_0=0$ corresponding to the radius $r_0=1$, we switch from an autonomous system (${\rm S}_{l_1}$),
 $g_{l_1}(x)=K_1 x|x|^{q_1-2}$ to another autonomous system (${\rm S}_{l_2}$), and $g_{l_2}(x)=K
 _2 x|x|^{q_2-2}$ where $K_1K_2<0$.
 The existence of \sol R{}f of equation~\eqref{rad.hardy} is given by the existence of homoclinic orbits which can be found in correspondence of intersections between the unstable manifold of system (${\rm S}_{l_1}$) and the stable manifold of system (${\rm S}_{l_2}$), see Figure~\ref{overlap}.

 \begin{figure}[t]
\centerline{\epsfig{file=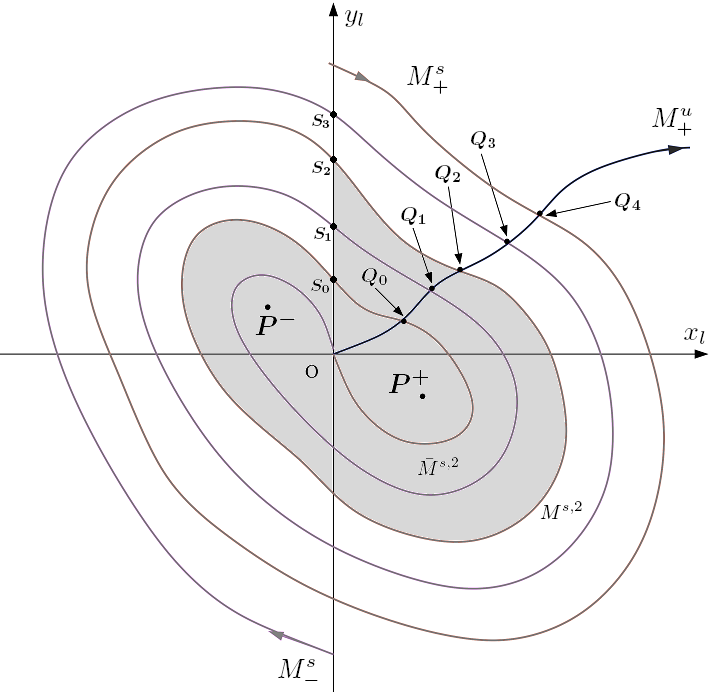, width = 8 cm}}
\caption{\small
Overlapping the phase portraits of  (${\rm S}_{l_1}$) and (${\rm S}_{l_2}$) we can find an infinite number of intersections respectively between the unstable manifold $M^u_+$ of (${\rm S}_{l_1}$) and  stable manifolds $M^s_\pm$ of (${\rm S}_{l_2}$). The points $\bs {Q_j}$ correspond to homoclinic orbits of~\eqref{sist*}, and so to \sol R{}f  of equation~\eqref{rad.hardy} under the assumptions of Theorem~\ref{main1}. For illustrative purpose, we have signed also the points $\bs{S_j}$ and we have coloured the region $\bar M^{s,2}$: its boundary consists of the branch $M^{s,2}\subset M^s$ and the segment between $\bs{S_0}$ and $\bs{S_2}$.
}
\label{overlap}
\end{figure}

\medbreak

Let us denote by $\phi^*(t,\Q)$, $\phi^u(t,\Q)$,  $\phi^s(t,\Q)$ respectively  the solutions of systems
(\ref{sist*}), (${\rm S}_{l_1}$), (${\rm S}_{l_2}$).
\begin{lemma}\label{stella}
Under the assumption of Theorem~\ref{main1bis}, we find an infinite number of intersections,
 say $\bs{Q_j}$, $j=0,1, \ldots$, between the unstable manifold $M^u$ of (${\rm S}_{l_1}$) and the stable manifold $M^s$ of (${\rm S}_{l_2}$). Further $\bs{Q_j}\in M^s_+$
if $j$ is even and $\bs{Q_j}\in M^s_-$
if $j$ is odd; moreover $\phi^*(t, \bs{Q_j})=\phi^s(t, \bs{Q_j})$ performs an angle $ -(j+1) \pi <\theta_j < -j \pi$
for $t \ge 0$.
\end{lemma}
\begin{proof}
Recall that $M^s_+$ and $M^s_-$ are two spirals rotating  counter-clockwise around the origin, crossing the coordinate axis transversally infinitely many times,
 see Section~\ref{mag0}.
Let $\bs{S_0}$, $\bs{S_{2j}}$ be the first and the $(j+1)^{\rm{th}}$ intersection
of $M^s_+$ with the $y$ positive semi-axis,
 and let $\bs{S_1}$, $\bs{S_{2k+1}}$ be the first and the $(k+1)^{\rm{th}}$ intersection
of $M^s_-$ with the $y$ positive semi-axis (see Figure~\ref{overlap}). Denote respectively by
$M^{s,0}$, $M^{s,2k}$  the branches of $M^s_+$ respectively between the origin and
$\bs{S_0}$, and between $\bs{S_{2k-2}}$ and $\bs{S_{2k}}$. Analogously denote by
$M^{s,1}$, $M^{s,2k+1}$  the branches of $M^s_-$ respectively between the origin and
$\bs{S_1}$, and between $\bs{S_{2k-1}}$ and $\bs{S_{2k+1}}$.
Denote by $\bar{M}^{s,0}$ the closed bounded set enclosed by
$M^{s,0}$ and the segment between the origin and $\bs{S_0}$, by $\bar{M}^{s,1}$ the closed bounded set enclosed by
$M^{s,1}$ and the segment between the origin and $\bs{S_1}$, by $\bar{M}^{s,j}$ the closed bounded set enclosed by
$M^{s,j}$ and the segment between $\bs{S_j}$ and $\bs{S_{j-2}}$ for $j \ge 2$, see Figure~\ref{overlap}.

Observe further that  $M^u_+$ is contained in $T^+(\kappa(\eta))$ and it is a graph on the $x$ positive semi-axis,
see Lemma~\ref{pictureneg}.
Let $\gamma^{u,\pm}:[0,+\infty)\to\RR^2$, $\gamma^{s,\pm}:[0,+\infty)\to\RR^2$ be the arc length parametrizations
 respectively  of $M^{u}_\pm$, $M^{s}_\pm$, with $\gamma^{u,\pm}(0)=\gamma^{s,\pm}(0)=(0,0)$; we introduce  polar coordinates as follows
$$\begin{array}{c}
    \gamma^{u,\pm}(\upsilon)=\varrho^{u,\pm}(\upsilon) (\cos \theta^{u,\pm}(\upsilon), \sin\theta^{u,\pm}(\upsilon))\,, \\
\gamma^{s,\pm}(\sigma)=\varrho^{s,\pm}(\sigma) (\cos \theta^{s,\pm}(\sigma), \sin\theta^{s,\pm}(\sigma))\,.
  \end{array}
$$
with  $\theta^{u,+}(0)= -\arctan(\kappa(\eta))$,   $\theta^{s,+}(0)= -\arctan(n-2-\kappa(\eta))$,
$\theta^{u,-}(0)= \pi-\arctan(\kappa(\eta))$   $\theta^{s,-}(0)= \pi-\arctan(n-2-\kappa(\eta))$.
We have $\theta^{u,+}(\upsilon)\in(-\pi/2,\pi/2)$, $\theta^{u,-}(\upsilon)\in(\pi/2,3\pi/2)$ for every $\upsilon$; moreover $\varrho^{u,\pm}$, $\theta^{s,\pm}$ and $\varrho^{s,\pm}$  diverge to infinity as $\upsilon,\sigma\to\infty$.
From a simple reasoning, it follows that for any $j \in \mathbb{N}$ there is  $u_j$ such that $\gamma^{u,+}(\upsilon_j)=\bs{Q_j} \in M^{s,j}$ and
$\gamma^{u,+}(\upsilon) \in \bar{M}^{s,j}$ for any $0<\upsilon< \upsilon_j$, and $\phi^*(t, \bs{Q_j})$ has the desired properties.
\end{proof}

 In the previous proof we have denoted by $\bs{Q_j}$   the first intersection point of $M^u_+$ with $M^{s,j}$ (with respect to the arc length parametrization of $M^u_+$).  Let
$$
\upsilon^{*}_j:= \inf \{ \tilde \upsilon \in (\upsilon_{j-1},\upsilon_j) \mid \gamma^{u,+}(\upsilon) \in (\bar{M}^{s,j} \setminus \bar{M}^{s,{j-1}}) \,, \forall \upsilon \in (\tilde \upsilon,\upsilon_j)\}\,.
$$
 We denote by
$\bs{Q_{j-1}^*}=\gamma^{u,+}(\upsilon_{j}^*)$   the last intersection point between $M^{s,j-1}$ and $M^u_+$ ``before'' $\bs{Q_j}$.
Notice that  if $M^{s,j}\cap M^u_+$ consists only of the point $\bs{Q_j}$ we have  $\bs{Q_j}=\bs{Q_j^*}$.
Unfortunately, we cannot prove that $M^{s,j}\cap M^u_+=\{\bs{Q_j}\}$:  such a ``uniqueness'' property can be obtained under the additional assumptions of Theorem~\ref{main1bis}, which guarantee the validity of Lemma~\ref{transversal} below.

\medbreak

\noindent
\emph{Proof of Theorem~\ref{main1}.}
Let $\bs{Q^u}=\gamma^{u,+}(\upsilon)$ where  $\upsilon \in [\upsilon^*_j, \upsilon_j]$;
then $\phi^*(t,\bs{Q^u}) \in M^u_+ \subset T^+(\kappa(\eta))$ for any $t \le 0$, and $\phi^*(t,\bs{Q^u})\in \bar M^{s,j}$ for every $t>0$,
cf. Lemma~\ref{pictureneg}.

In particular $\gamma^{u,+}(\upsilon^*_j) \in M^{s,j-1}$ and $\gamma^{u,+}(\upsilon_j) \in M^{s,j}$ so that, when $t>0$, the solutions $\phi^*(t,\gamma^{u,+}(\upsilon^*_j))$ and $\phi^*(t,\gamma^{u,+}(\upsilon_j))$ follow the stable manifold $M^s$ towards the origin intersecting transversally the $y$ axis respectively $j-1$ and $j$ times. Thus we have two  \sol R{}f's  $u(r,\tilde D_k)$ and $u(r,D_k)$ with $D_{k-1}\leq \tilde D_k<D_k$ (the correct order is given by Lemma~\ref{manifold}) with respectively $j-1$ and $j$ nondegenerate zeros.

Consider now $\bs{Q^u}=\gamma^{u,+}(\upsilon)$ with  $\upsilon \in (\upsilon^*_j, \upsilon_j)$. In particular $\phi^*(t,\bs{Q^u})\notin M^s$ for every $t>0$.
 The solution $\phi^*(t,\bs{Q^u})$, forced to belong to $\bar M^{s,j}$ for $t>0$, is attracted towards $\bs{P^+}$ if $j$ is even, towards $\bs{P^-}$ if $j$ is odd, and, {\em guided} by the stable manifold $M^s$, crosses exactly $j$ times (transversally) the $y$ axis for $t > 0$.

Therefore the  corresponding solution $u(r,d^\upsilon)$ is a \sol R{}s which is positive for $r \le 1$ and changes sign exactly $j$ times for $r>1$.
\qed

\medbreak

The additional assumption required by Theorem~\ref{main1bis} forces, roughly speaking, the  unstable manifold $M^u$ to intersect the stable manifold $M^s$ passing from the inner part to the outer part of the spiral, thus giving a kind of uniqueness result.

\begin{lemma}\label{transversal}
Assume that we are in the hypothesis of Theorem~\ref{main1bis}.
Then $\gamma^{u,+}(\upsilon) \not\in \bar{M}^{s,j}$ for any $\upsilon>\upsilon_j$.
\end{lemma}
\begin{proof}
We claim that {\bf the flows  of $({\rm S}_{l_1})$ and  $({\rm S}_{l_2})$ are transversal in any point $\Q \in T^+(\kappa(\eta))$}.
Let  $J = \left( \begin{array}{cc} 0&-1 \\ 1&0 \end{array}\right)$ and
$\boldsymbol Q = (x_Q,y_Q) \in T^+(\kappa(\eta))$, then we have
two solutions $\boldsymbol{ x_{l}^u}$ of (${\rm S}_{l_1}$) and $\boldsymbol{ x_{l}^s}$ of (${\rm S}_{l_2}$)
 passing through $\boldsymbol Q$ at $t=\tau_0$, i.e. $\boldsymbol Q =\boldsymbol{ x_{l}^u}(\tau_0)=\boldsymbol{ x_{l}^s}(\tau_0)$. We can compute the following scalar product
\begin{equation}\label{toout}
\begin{split}
& \left\langle \boldsymbol{\dot x_{l}^u}(\tau_0) \,,\, J \boldsymbol{\dot x_{l}^s}(\tau_0) \right\rangle =
(\alpha_{l_1}-\alpha_{l_2})[(n-2)x_Q y_Q + y_Q^2+ \eta x_Q^2]
\\ & \qquad\qquad\qquad -(\alpha_{l_2} x_Q + y_Q)K_1 x|x|^{q_1-2}+(\alpha_{l_1} x_Q + y_Q)K_2 x|x|^{q_2-2}
\end{split}
\end{equation}
Notice that if $\Q \in T(m)$, i.e. $y_Q=-m x_Q$ then
$$[y_Q^2+(n-2)x_Q y_Q +  \eta x_Q^2]= x^2[ m^2-(n-2)m +\eta] \ge 0$$
whenever $m \le \kappa(\eta)$. Since $\Q\in T^+(\kappa(\eta))$,
 remembering that $K_1<0$, and both $\alpha_{l_1},\alpha_{l_2}\geq \kappa(\eta)$ by the assumption $l_1\le l_2\leq {\rm I}(\eta)$;
 from~\eqref{toout} we find
$\left\langle \boldsymbol{\dot x_{l}^u}(\tau_0) \,,\, J \boldsymbol{\dot x_{l}^s}(\tau_0) \right\rangle>0$.
Hence, the claim is proved.

Assume for contradiction that there is $\tilde{\upsilon}>\upsilon_j$ such that $\gamma^{u,+}(\tilde{\upsilon}) \in \bar{M}^{s,j}$,
then there is  $\bar{\upsilon} \in (\upsilon_j, \tilde{\upsilon}]$ such that $\gamma^{u,+}(\bar{\upsilon})=\bs{\bar{Q}} \in M^{s,j}$ and
$\gamma^{u,+}(\upsilon)\not\in M^{s,j}$ for any $\upsilon \in (\upsilon_j, \bar{\upsilon})$, i.e. the manifold $M^{u}_+$ exit from $\bar{M}^{s,j}$
in $\bs{Q_j}$ and enters again in $\bar{M}^{s,j}$ in $\bs{\bar{Q}}$. It follows that
$\left\langle \boldsymbol{\dot \phi}^u (0,\bs{\bar{Q}}) \,,\, J \boldsymbol{\dot \phi}^s (0,\bs{\bar{Q}}) \right\rangle \leq 0$
which contradicts~\eqref{toout} since $\bs{\bar{Q}}\in M^{u}_+ \subset T_+(\kappa(\eta))$, so the Lemma is proved.
\end{proof}

\emph{Proof of Theorem~\ref{main1bis}.}
The proof follows step by step the one of Theorem~\ref{main1}. Then, as a straightforward consequence of the previous lemma, we find $\tilde D_k=D_{k-1}$ for every $k\geq 1$, thus completing the proof.
\qed

\medbreak

%\ft{MF: questo eviterei di dirlo\\
%It is worthwhile to mention that the assumptions of Theorem~\ref{main1bis} can be weakened. In fact, we have proved Lemma~\ref{transversal} proving that all the terms of the right hand side of~\eqref{toout} are positive. We do not enter in such details for briefness. By the way, we underline that Theorem~\ref{main1bis} holds for equation~\eqref{laplace} whenever Theorem~\ref{main1} applies.}

\medbreak

The proof of Theorems~\ref{main2} and~\ref{main2bis} can be obtained similarly: in this case the unstable manifold $M^u$  of (${\rm S}_{l_1}$) consists of a double spiral, while the stable manifold $M^s$ of (${\rm S}_{l_2}$) is unbounded, cf. Figures~\ref{posportrait} and~\ref{negportrait}.
The proof is then obtained arguing as in  Theorems~\ref{main1} and~\ref{main1bis},
but reversing time:
we leave  details to the reader. By the way, the proof can be obtained immediately also by the use of Kelvin inversion, see e.g.~\cite{DF,Fdie}.

Corollaries~\ref{cor1} and~\ref{cor2} are direct consequences of Theorems~\ref{main1bis} and~\ref{main2bis}.

\section{More on applications}\label{moreonappl}

\subsection{A slight generalization}\label{slight}
The whole analysis performed in Section~\ref{due} can be trivially extended to a
slightly more general family of potentials $f$. Let us denote by
\begin{equation}\label{defg}
K g_l(x,t)=f(x \eu^{-\alpha_l t}, \eu^t) \eu^{(\alpha_l+2)t} \, ,
\end{equation}
where $K \ne 0$ is a constant, we introduce the following assumption:
\begin{description}
\item[$\bs{G0}$] There is $l>2$ such that  $g_l(x,t)=g_l(x)$ is $t$-independent.
Further $g_l(0)=g'_l(0)=0$, $g_l(x)/x$ is positive, decreasing for $x<0$
and increasing for $x>0$ and $\lim_{x \to \pm \infty} g_l(x)/x=+ \infty$.
\end{description}
Note that morally we are assuming that $g_l$ is autonomous, convex for $x>0$ and it is odd.
If we apply~\eqref{transf} to
~\eqref{rad.hardy} we obtain again~\eqref{sist} where $g_l(x)$ replaces
 $ x|x|^{q-2}$.
The convexity condition requiring  that $g_l(x)/x$ is decreasing for $x<0$ and increasing for $x>0$ is needed in order to ensure the existence of the critical points $\bs {P^\pm}$, therefore it may be dropped in the case  where this point does not exist (i.e. either  $K>0$ and
 $l \in (2, 2_*(\eta)] \cup [I(\eta),+\infty)$ or $K<0$ and $l \in [2^*(\eta),I(\eta)]$).

 A class of $f$ which fits $\bs{G0}$ (besides of~\eqref{auto}) is given
by the following
\begin{equation}\label{strangef}
f(u,r)= K_1 r^{\delta_1}u|u|^{q_1-2}+K_2 r^{\delta_1}u|u|^{q_2-2}+
K_3 r^{\delta_3} u|u|^{q_3-2} \ln(1+|u|r^{\delta})
\end{equation}
where $q_1,q_2>2$, $q_3\geq2 $,  $K_i \ge 0$ for $i=1,2,3$, $\sum K_i^2 >0$, $2 \frac{q_i+\delta_i}{2+\delta_i}=l$ for $i=1,2,3$  and $\delta=\alpha_l$.

In fact we can also assume that $f$ is not odd in $u$; we may even
  have two different values of $l$, say $l_a, l_b$, the former for
  $u$ positive, the latter for $u$ negative.
  However we need $l_a, l_b$ to belong to the same range of parameters, i.e.
  $l_a,l_b \in (2_*(\eta), I^*(\eta))$, and
$g_{l_a}(x,t)$ and $g_{l_b}(x,t)$ are $t$ independent respectively for $x\le 0$
and for $x \ge 0$. We do not enter in more details for briefness.

\subsection{Some further Remarks}\label{furtherrem}

Let us consider the following equation.
Let $\rho>0$ and consider the following generalization of \eqref{rad.hardy}:
 \begin{equation}\label{types2}
 \begin{array}{l}
\ds u''+ \frac{n-1}{r}\, u'+\frac{\eta}{r^2}\, u + f(u,r)=0\,, \\[2mm]
\hspace{20mm}
f(u,r)=
 \begin{cases}
K u |u|^{q_1-2} & r \le \rho \\
-K  u |u|^{q_2-2} & r>\rho \,.
 \end{cases}
 \end{array}
 \end{equation}
We recall that Theorems \ref{main1}, \ref{main1bis}, \ref{main2}, \ref{main2bis} hold in this context too, as we specified in the introduction.
The analysis can be easily extended to embrace the general case of $f$ of type \eqref{types},
but have decided to restrict our attention to~\eqref{types2} to make the argument more transparent.

Assume for definiteness that we are in the hypotheses of Theorem \ref{main1}, so that there exists a \sol R{0}f solution $u(r,D_0)$.
We denote by $R_0(\rho,K)$ the value of $r$ such that $u(r,D_0)$ attains its maximum,  and by $U_0(\rho,K)=u(R_0(\rho,K), D_0)$ the maximum itself. We will denote by $D_0(\rho,K)$ the value $D_0$ in order to emphasize its dependence of the parameters $\rho$ and $K$.
In this subsection we want to establish the relationship between such parameters  and the values
$R_0(\rho,K)$, $U_0(\rho,K)$, $D_0(\rho,K)$.
These results are elementary but may be of use from an application point of view, especially for their simplicity.
Remember that $K$ represents the ratio between the velocity of the reaction and of the diffusion, while $\rho>0$ represents
the size of the set where we have production.

From a straightforward computation, we easily get the 
following well known scaling property of equation~\eqref{types2}.

\begin{remark}\label{scalingrem}
Let us consider a radial solution $u(r)$ of~\eqref{types2} where $K=1$ and $\rho=1$, then
 $w(r)=  [\bar{\rho}^{2} \bar{K}]^{-1/(q_1-2)}u(r/\bar\rho)$ is a
radial solution of~\eqref{types2} where $K=\bar{K}$, and $\rho=\bar{\rho}>0$.
\end{remark}

%\ft{AS: In fact if $w(s)= A u(s/\bar\rho)$ we have
%$$
%\begin{array}{rl}
%\ddot w +\frac{n-1}{r} \dot w +\frac{\eta}{r^2}  w
%& =  \frac{A}{\bar\rho^2} \left[ u'' + \frac{n-1}{r} u' +\frac{\eta}{r^2} u \right]_{r=s/\bar\rho}\\
%& =  \frac{A}{\bar\rho^2} \left[ - u|u|^{q_1-2} \right]_{r=s/\bar\rho}\\
%& =  \frac{A}{\bar\rho^2} \frac{1}{A^{q-1}} \left[ - w|w|^{q_1-2} \right]\\
%& = K \left[ - w|w|^{q_1-2} \right]\,. \quad K=[\bar\rho^2 A^{q-2}]-1\\
%\end{array}
%$$
%so that $ A=  [\bar{\rho}^{2} \bar{K}]^{-1/(q_1-2)}$}
%Let $u(r,D_0)$ be the ground state with fast decay given either by Theorem~\ref{main1bis} or by Theorem~\ref{main2bis};
%we denote by $U_0>0$   the unique local maximum of $u(r,D_0)$ and by $r=R_0$ be value at which it is attained.
%Consider equation \eqref{types2} and assume the hypotheses  either of Theorem~\ref{main1bis} or of Theorem~\ref{main2bis}.
%Let $u(r,D_0)$ be the unique ground state with fast decay
%we denote by $U_0>0$   the unique local maximum of $u(r,D_0)$ and by $r=R_0$ be value at which it is attained.
\noindent
Then, using the previously introduced notations, we get
\begin{equation}\label{pp}
\begin{array}{l}
         R_0(K,\rho)=\rho R_0(1,1) \, , \\[1mm]
         U_0(K,\rho)= [ \rho^{2} K]^{-1/(q_1-2)} U_0(1,1) \, , \\[1mm]
         D_0(K,\rho)= [\rho^{2} K]^{-1/(q_1-2)} D_0(1,1)  \, .
 \end{array}
\end{equation}

Assume that we are in the setting of Theorem~\ref{main1bis} (respectively of Theorem~\ref{main2bis}).
Estimates in~\eqref{pp} show explicitly that the maxima $U_0$  decrease with the size of the bounded region where
we have production (respectively absorption), and also if diffusion gets stronger.
I. e., the ground states gets more concentrated and have larger maxima if the bounded region is smaller.
The same happens to the initial conditions
$D_0$, while the value $R_0$ at which the maxima is attained is not influenced by the ratio
between strength of the reaction and diffusion.

We think it is worth observing that the dependence of $U_0$ on $K$, and $\rho$ does not change if we have absorption
for $r \le \rho$ and production outside, as in Theorem \ref{main1bis}, or in the opposite situation, as in Theorem \ref{main2bis}.

%%%%%%%%%%%%%%%%%%%%%%%%%

\end{document}